\pdfoutput=1
\RequirePackage{ifpdf}
\ifpdf 
\documentclass[pdftex]{sigma}
\else
\documentclass{sigma}
\fi

\DeclareMathOperator{\lap}{\Delta}
\DeclareMathOperator{\tr}{tr}

\newcommand{\mexhal}{{\rm e}^{h(\alpha-1)}}
\newcommand{\exhal}{{\rm e}^{-h(\alpha-1)}}
\newcommand{\exh}{{\rm e}^{-h}}

\DeclareMathOperator{\divg}{div}
\DeclareMathOperator{\grad}{\nabla}

\DeclareMathOperator{\n}{\mathbf{n}}

\newcommand{\N}{\mathbb{N}}

\newcommand{\R}{\mathbb{R}}

\newcommand{\dm}{\mathrm{d}m}
\newcommand{\dvg}{\mathrm{d}v_g}
\newcommand{\Ric}{\mathrm{Ric}}

\begin{document}

\allowdisplaybreaks

\newcommand{\arXivNumber}{2002.03698}

\renewcommand{\thefootnote}{}

\renewcommand{\PaperNumber}{090}

\FirstPageHeading

\ShortArticleName{About Bounds for Eigenvalues of the Laplacian with Density}

\ArticleName{About Bounds for Eigenvalues of the Laplacian\\ with Density\footnote{This paper is a~contribution to the Special Issue on Scalar and Ricci Curvature in honor of Misha Gromov on his 75th Birthday. The full collection is available at \href{https://www.emis.de/journals/SIGMA/Gromov.html}{https://www.emis.de/journals/SIGMA/Gromov.html}}}

\Author{A\"issatou Moss\`ele NDIAYE}

\AuthorNameForHeading{A.M.~Ndiaye}

\Address{Institut de Math\'ematiques, Universit\'e de Neuch\^atel, Switzerland}
\Email{\href{mailto:aissatou.ndiaye@unine.ch}{aissatou.ndiaye@unine.ch}}

\ArticleDates{Received February 13, 2020, in final form September 01, 2020; Published online September 25, 2020}

\Abstract{Let $M$ denote a compact, connected Riemannian manifold of dimension $n\in\N$. We assume that $ M$ has a smooth and connected boundary. Denote by $g$ and ${\rm d}v_g$ respectively, the Riemannian metric on $M$ and the associated volume element. Let $\Delta$ be the Laplace operator on $M$ equipped with the weighted volume form ${\rm d}m:= {\rm e}^{-h}\,{\rm d}v_g$. We are interested in the operator $L_h\cdot:={\rm e}^{-h(\alpha-1)}(\Delta\cdot +\alpha g(\nabla h,\nabla\cdot))$, where $\alpha > 1$ and $h\in C^2(M)$ are given. The main result in this paper states about the existence of upper bounds for the eigenvalues of the weighted Laplacian $L_h$ with the Neumann boundary condition if the boundary is non-empty.}

\Keywords{eigenvalue; Laplacian; density; Cheeger inequality; upper bounds}

\Classification{35P15; 58J50}

\renewcommand{\thefootnote}{\arabic{footnote}}
\setcounter{footnote}{0}

\section{Introduction}
Let $(M, g)$ be a compact connected $n$-dimensional Riemannian manifold. Let $h\in C^2(M)$ and~$\rho$ be the positive function define by $\rho:={\rm e}^{-h}$. Let $\dvg$, $\lap$ and $\nabla$ denote respectively, the Riemannian volume measure, the Laplace and the gradient operator on $(M,g)$. For simplicity, we also denote by $\dvg$ the volume element for the induced metric on $\partial M$. We define the Laplacian with negative sign, that is the negative divergence of the gradient operator.

The Witten Laplacian (also called drifting, weighted or Bakry--Emery Laplacian) with respect to the weighted volume measure $\rho\dvg$ is define by
\begin{equation}\label{witten}
\lap^W:=\lap\cdot + g(\grad h,\grad\cdot).
\end{equation}
We designate by $\{\lambda_k(\rho, \rho)\}_{k\geqslant 0}$ the spectrum of the operator in \eqref{witten} under Neumann conditions if the boundary is non-empty. Let $S_k$ be the set of all $ k$-dimensional vector subspaces of $H^1(M)$, the spectrum consists of a non-decreasing sequence of eigenvalues variationally defined by
\begin{equation}\label{cwe}
\lambda_k(\rho, \rho)=\inf_{V\in S_{k+1}}\sup_{u\in V\backslash \{0\}}\frac{\int_{M} |\nabla u|^2\rho \, \dvg}{\int_{M} u^2\rho \, \dvg},
\end{equation}
for all $k\geqslant 0$.

In recent years, the Witten Laplacian received much attention from many mathematicians (see \cite{du2017,du2016,nier2005hypoelliptic,huang2011,li2012perelman, lu2010, ma2008, xia2014} and the references therein), in particularly the classical research topic of estimating eigenvalues.

When $h$ is a constant, the Witten Laplacian is exactly the Laplacian.
Another spectrum has a similar characterisation with the one of the Witten Laplacian: the spectrum of the Laplacian associated with the metric $\rho^{\frac{2}{n}}g$, which is conformal to $g$. It is natural to denote its spectrum by $\big\{\lambda_k\big(\rho, \rho^{\frac{n-2}{n}}\big)\big\}_{k\geqslant 0}$, since the eigenvalues are variationally characterised by
\begin{equation}\label{cce}
\lambda_k\big(\rho, \rho^{\frac{n-2}{n}}\big)=\inf_{V\in S_{k+1}}\sup_{u\in V\backslash \{0\}}\frac{\int_{M} |\nabla u|^2\rho^{\frac{n-2}{n}} \,\dvg}{\int_{M} u^2\rho\, \dvg}.
\end{equation}

In the present work, we are interested in the expanded eigenvalue problem of the Dirichlet energy functional weighted by
$\rho^\alpha$, with respect to the $L^2$ inner product weighted by $\rho$, where $\alpha\geqslant 0$ is a given constant. These eigenvalues are those of the operator $L_{\rho}\cdot= L_h\cdot:=-\rho^{-1}\divg (\rho^\alpha\grad \cdot )= {\rm e}^{-h(\alpha-1)} ( \lap\cdot +\alpha g(\grad h,\grad\cdot) )$
on $M$ endowed with the weighted volume form $\dm:=\rho\,\dvg$. The spectrum consists of an unbounded increasing sequence of eigenvalues
\[ \operatorname{Spec}( L_h) = \big\{0 = \lambda_0(\rho,\rho^\alpha) < \lambda_1(\rho,\rho^\alpha) \leqslant \lambda_2 (\rho,\rho^\alpha) \leqslant \cdots \leqslant \lambda_k(\rho,\rho^\alpha) \leqslant \cdots \big\},\]
 which are given by
\[ \lambda_k(\rho,\rho^\alpha)=\inf_{V\in S_{k+1}}\sup_{u\in V\backslash \{0\}}\frac{\int_{M} |\nabla u|^2\rho^\alpha\, \dvg}{\int_{M} u^2\rho\, \dvg},\]
for all $k\geqslant 0$. As already mentioned, $S_k$ is the set of all $k$-dimensional vector subspaces of~$H^1(M)$. The particular cases where $\alpha=1$ and $\alpha=\frac{n-2}{n}$ correspond to the problems mentioned above, whose eigenvalues are respectively given by~\eqref{cwe} and~\eqref{cce}.

A main interest is to investigate the interplay between the geometry of $(M, g)$ and the effect of the weights, looking at the behaviour of $\lambda_k(\rho,\rho^\alpha)$, among densities $ \rho$ of fixed total mass.
The more general problem where the Dirichlet energy functional is weighted by
a positive function~$\sigma$, not necessarily related to $\rho$ is presented by Colbois and El~Soufi in \cite{Colbois2019}.

In the aforementioned paper, Colbois and El~Soufi exhibit an upper bound for the singular case where $\alpha=0$ \cite[Corollary~4.1]{Colbois2019}:
\[ \lambda_k(\rho,1)|M|^\frac{2}{n}\leqslant C_nk^{\frac{2}{n}},\]
where $C_n$ depends only on the dimension $n$. Whereas, in \cite[Theorem~5.2]{colboissoufisavo}, Colbois, El Soufi and Savo prove that, when $\alpha=1$, there is no upper bound among all manifolds. Indeed, they show that, on a compact revolution manifold, one has $\lambda_1(\rho,\rho)$ as large as desired. In their work in~\cite{LucSalam}, Kouzayha and Pétiard give an upper bound for $\lambda_k(\rho,\rho^\alpha)$, when $\alpha\in\big(0,\frac{n-2}{n}\big]$ and prove that there is none for $\lambda_1(\rho,\rho^\alpha)$ when $\alpha$ runs over the interval $\big(\frac{n-2}{n}, 1\big)$.

In this work, we treat the remaining cases, that is when $\alpha>1$. We prove, as conjectured in \cite[Remark~3]{LucSalam}, that there is no upper bound for $\lambda_1(\rho,\rho^\alpha)$, in the class of manifolds $M$ with convex boundary and positive Ricci curvature.
 \begin{theorem}\label{main1}
Let $\alpha>1$ be a given real constant. Let $(M, g)$ be a compact connected Riemannian manifold of dimension $n$, whose Ricci curvature satisfies $\Ric\geqslant \kappa$, for some positive constant~$\kappa$. If~$M$ has convex boundary, then there exists a sequence of densities $\{\rho_j\}_{ j\geqslant 2}$ and $j_0\in \N$, such that
\[\lambda_1(\rho_j,\rho_j^\alpha)\left(\frac{|M|}{\int_M \rho_j\dvg} \right)^{\alpha-1}\geqslant 2\kappa j,\qquad \forall\, j\geqslant j_0. \]
Here, $|M|$ denotes the volume of $M$.
\end{theorem}

 This inequality provides a lower bound that grows linearly to infinity in $j$ as $j \rightarrow \infty$, showing that with respect to these densities, $ \lambda_1(\rho,\rho^\alpha)$ becomes as large as desired.
Unfortunately, I do not know any other way to prove it, than the following long and painful computation.

Our aim is to show that, there exists a family of densities $\rho_j={\rm e}^{-h_j}$, $j\in \N$, such that their corresponding first non-zero eigenvalues become as large as desired. For this, we use the extended Reilly formula presented in Theorem~\ref{thmreilly}, to provide a lower bound that grows linearly to infinity in $j$, as $j \rightarrow \infty$.

Let $(M,g)$ be a compact connected Riemannian manifold of dimension $n$ with smooth boundary $\partial M$. Let $D^2$ denote the Hessian tensor, $\grad_{\partial}$ the tangential gradient, $\lap_{\partial}$ the Laplace--Beltrami operator on $\partial M$ and $\partial_{\n}$ the derivative with respect to the outer unit normal vector ${\n}$ to $\partial M$.
The second fundamental form on $\partial M$ is defined by $I(X, Y ):= g(\nabla_X \n, Y )$ for any vector fields~$X$ and $Y$. Let $H := \tr I$ denote the mean curvature of $\partial M$ and $\Ric$ the Ricci curvature on~$M$.

\begin{theorem}[Reilly formula]\label{thmreilly}
Consider $M$ equipped with the weighted volume form $\dm =\exh\, \dvg$ for some $h\in C^2(M)$. Then, for every $u \in C^\infty(M)$, we have:
\begin{gather}
\int_M\mexhal| L_h u|^2-\exhal |D^2u|^2\,\dm
=\int_M\exhal \left( \Ric+\alpha D^2h\right)\left(\grad u,\grad u\right)\dm\nonumber\\
\qquad{}+\int_{\partial M}\exhal g\big( \partial_{\n} u,H\partial_n u-\alpha g(\grad h,\grad u)-\lap_{\partial } u\big)\,\dm\nonumber\\
\qquad{} +\int_{\partial M}\exhal \big[I(\grad_{\partial }u,\grad_{\partial }u)-g(\grad_\partial u,\grad_\partial \partial_{\n} u)\big]\,\dm.
\label{reilly}
\end{gather}
\end{theorem}
In the next section, we prove these two theorems.

\section{Proofs}
\subsection{Proof of Theorem \ref{thmreilly}}

To prove Theorem \ref{thmreilly}, one needs the following adapted Bochner formula deduced from the standard one for smooth functions (see, e.g., \cite[Theorem~346]{berger2007panoramic} and~\cite{Bochner1946}):
\begin{equation*}
\frac{1}{2} \lap \big(|\nabla u|^2\big)=-|D^2u|^2 +g(\grad u,\grad\lap u)-\Ric(\grad u,\grad u).
\end{equation*}

\begin{lemma}Let $u$ be a smooth function on $(M,g)$. Then,
\begin{gather}
\frac{1}{2} L_h |\nabla u|^2=-\exhal \big(|D^2u|^2+\big(\Ric+\alpha D^2h\big)(\grad u,\grad u)\big)\nonumber\\
\hphantom{\frac{1}{2} L_h |\nabla u|^2=}{} +g\big(\grad u,\grad L_h u+(\alpha-1)g(\grad h, L_h u)\big). \label{bochner}
\end{gather}
\end{lemma}
\begin{proof}
\begin{gather*}
\frac{1}{2} L_h |\nabla u|^2 =\frac{1}{2}\exhal \big( \lap |\nabla u|^2 +\alpha g\big(\nabla h, \nabla |\nabla u|^2\big)\big)\\
\hphantom{\frac{1}{2} L_h |\nabla u|^2}{} =\exhal \big({-}\big|D^2 u\big|^2+g(\grad u,\grad\lap u)-\Ric(\grad u,\grad u)\big)\\
\hphantom{\frac{1}{2} L_h |\nabla u|^2=}{} +\frac{1}{2}\alpha\exhal g\big(\nabla h, \nabla |\nabla u|^2\big)\\
\hphantom{\frac{1}{2} L_h |\nabla u|^2}{} =-\exhal \big(\big|D^2 u\big|^2+\Ric(\grad u,\grad u)\big) +\exhal g(\grad u,\grad\lap u)\\
\hphantom{\frac{1}{2} L_h |\nabla u|^2=}{} -\alpha\exhal D^2h(\grad u,\grad u)+\alpha\exhal g ( \grad (g(\grad h,\grad u) ),\grad u ).
\end{gather*}
For the last line, we have used
\begin{gather*}
\frac{1}{2}g\big(\nabla h, \nabla |\nabla u|^2\big)=g(\nabla h, \grad_{\grad u}\grad u)=
D_{\grad u}g(\grad h,\grad u) -g(\grad_{\grad u}\grad h,\grad u)\\
\hphantom{\frac{1}{2}g\big(\nabla h, \nabla |\nabla u|^2\big)}{}
=g ( \grad (g(\grad h,\grad u) ),\grad u )- D^2h(\grad u,\grad u).
\end{gather*}
Moreover,
\begin{gather*}
g (\grad( L_h u),\grad u )=-(\alpha-1)g ( g(\grad h, L_h u), \grad u )\\
\hphantom{g (\grad( L_h u),\grad u )=}{} +\exhal g ( \grad\lap u, \grad u )+\alpha\exhal g (\grad (g(\grad h,\grad u) ), \grad u ).
\end{gather*}
Finally,
\begin{gather*}
\frac{1}{2} L_h |\nabla u|^2=-\exhal \big(\big|D^2u\big|^2+\big(\Ric+\alpha D^2h\big)(\grad u,\grad u)\big) \\
\hphantom{\frac{1}{2} L_h |\nabla u|^2=}{}
 +g\big(\grad u,\grad L_h u+(\alpha-1)g(\grad h, L_h u)\big).\tag*{\qed}
\end{gather*}\renewcommand{\qed}{}
\end{proof}

\begin{proof}[Proof of Theorem \ref{thmreilly}]
We shall integrate equality \eqref{bochner}. On the left-hand side, we have
\begin{gather*}
 \frac{1}{2}\int_M L_h|\grad u|^2 \,\dm =\frac{1}{2} \int_M \exhal \big(\lap|\grad u|^2+\alpha g \big(\grad h,\grad|\grad u|^2\big)\big)\,\dm\\
\hphantom{\frac{1}{2}\int_M L_h|\grad u|^2 \,\dm}{} =\frac{1}{2}\int_M g\big(\grad\big(|\grad u|^2\big),\grad\big({\rm e}^{-\alpha h}\big)\big)\,\dvg-\frac{1}{2}\int_{\partial M}\partial_{\n}\big(|\grad u|^2\big){\rm e}^{-\alpha h}\,\dvg \\
\hphantom{\frac{1}{2}\int_M L_h|\grad u|^2 \,\dm=}{} + \frac{1}{2} \alpha\int_M \exhal g\big(\grad h,\grad|\grad u|^2 \big)\,\dm \\
\hphantom{\frac{1}{2}\int_M L_h|\grad u|^2 \,\dm}{}
 =-\int_{\partial M}\exhal g (\partial_{\n}(\grad u),\grad u )\,\dm.
\end{gather*}
The second term on the right-hand side gives
\begin{gather*}
 \int_M g\big(\grad u,\grad L_h u+(\alpha-1) g(\grad h, L_hu) \big)\,\dm
=\int_M g \big( \grad u,\exh\grad L_h u \big) \,\dvg\\
\qquad \quad{} + (\alpha-1)\int_M g ( \grad u,g(\grad h, L_hu) )\exh\, \dvg \\
\qquad{} =\int_M g \big( \grad u,\grad \big( L_h u\exh\big) \big) \,\dvg+\int_M g (\grad u,g(\grad h, L_hu) )\,\dm\\
 \qquad\quad{}+(\alpha-1)\int_M g (\grad u,g(\grad h, L_hu) )\,\dm \\
 \qquad{} =\int_M g (\lap u, L_hu )\, \dm+\alpha\int_M g ( L_h u, g (\grad h,\grad u ) )\,\dm
 +\int_{\partial M} g (\partial_{\n} u, L_hu )\,\dm\\
\qquad{} =\int_M \mexhal| L_h u|^2 \dm+\int_{\partial M} g (\partial_{\n} u, L_h u ) \,\dm.
\end{gather*}
Then, replacing in \eqref{bochner}, one has
\begin{gather}
 -\int_{\partial M} \exhal g\big(\partial_{\n} (\grad u), \grad u\big)\,\dm\nonumber\\
\qquad{} =-\int_M\exhal \big|D^2u\big|^2\, \dm -\int_M\exhal \big(\Ric+\alpha D^2h\big) (\grad u,\grad u )\,\dm\nonumber\\
\qquad\quad{}
+\int_M\mexhal| L_hu|^2\,\dm+\int_{\partial M}g (\partial_{\n} u, L_h u )\, \dm,\nonumber\\
 \int_M\mexhal| L_hu|^2-\exhal \big|D^2u\big|^2\,\dm
=\int_M\exhal \big(\Ric+\alpha D^2h\big) (\grad u,\grad u )\,\dm\nonumber\\
\qquad\quad{} -\int_{\partial M}g(\partial_{\n}(\grad u),\grad u)\exhal +g (\partial_{\n} u, L_hu ) \,\dm.\label{eq03082020}
\end{gather}
Now, it remains to estimate $ [ g(\partial_{\n}(\grad u),\grad u)\exhal +g(\partial_{\n} u, L_hu) ] $ which is equal to
\begin{gather*}
\exhal [g(\partial_{\n}(\grad u),\grad u)+g(\partial_{\n} u,\lap u)+\alpha g(\partial_{\n} u,g(\grad h,\grad u) ].
\end{gather*}
 We notice that
\begin{gather}\label{eq19062020}
 \lap u=- H\partial_{\n} u+\lap_\partial u-\partial_{\n}^2u,
\end{gather}
 (see, e.g., \cite[equation~(3)]{casey}). We recall that our sign convention for the operators $\lap$ and $\lap_\partial$ is the opposite of that in \cite{casey}. Moreover,
 \begin{gather}
 g(\partial_{\n}(\grad u),\grad u)=(-\lap u+\lap_{\partial}u-H\partial_{\n} u)\partial_{\n} u -I(\grad_\partial u,\grad_\partial u)+g(\grad_\partial u,\grad_\partial \partial_{\n} u)\nonumber\\
 \hphantom{g(\partial_{\n}(\grad u),\grad u)}{} =g\big(\partial_{\n} u,\partial_{\n}^2 u\big)-I(\grad_\partial u,\grad_\partial u)+g(\grad_\partial u,\grad_\partial \partial_{\n} u)\label{eq03082020a}
 \end{gather}
(see \cite[p.~4]{casey}).

We then combine equalities \eqref{eq19062020} and \eqref{eq03082020a} to derive an expression for the last term in the right-hand side of~\eqref{eq03082020}:
\begin{gather*}
g (\partial_{\n}(\grad u),\grad u )\exhal +g(\partial_{\n} u , L_hu)
=\exhal [-I(\grad_\partial u, \grad_\partial u)+g(\grad_\partial u,\grad_\partial\partial_{\n} u) ]\\
\qquad{} +\exhal g (\partial_{\n} u ,-H\partial_{\n} u+\alpha g(\grad h,\grad u)+\lap_\partial u ).
\end{gather*}
Hence,
\begin{gather*}
\int_M\mexhal| L_h u|^2-\exhal \big|D^2u\big|^2 \dm
=\int_M\exhal \big(\Ric+\alpha D^2h\big) (\grad u,\grad u )\,\dm \\
\qquad{} +\int_{\partial M}\exhal g ( \partial_{\n} u,H\partial_{\n} u-\alpha g (\grad h,\grad u t)-\lap_\partial u )\,\dm\\
\qquad{} + \int_{\partial M}\exhal [I(\grad_\partial,\grad_\partial u)-g(\grad_\partial u,\grad_\partial\partial_{\n} u) ]\,\dm.\tag*{\qed}
\end{gather*}\renewcommand{\qed}{}
\end{proof}

\subsection{Proof of Theorem \ref{main1}}
Let $(M, g)$ be a compact connected $n$-dimensional Riemannian manifold with a convex boundary~$\partial M$. Let $h\in C^2(M)$ and assume that $\lambda$ is the first non-zero eigenvalue of $L_h$. Let $u\neq 0$ be an eigenfunction with corresponding eigenvalue $\lambda$, i.e., $u$ satisfies $L_h u=\lambda u$.

\begin{lemma}\label{lm07082020}
 If $\Ric +\alpha D^2h\geqslant \alpha^2\frac{|\grad h|^2}{nz}+A$, for some $A > 0$ and $z > 0$ then
\begin{gather}\label{eq3}
A\lambda\int_M u^2 \,\dm \leqslant \frac{\lambda^2}{n(z+1)}\int_M u^2 \big(\mexhal n(z+1)-\exhal \big)\,\dm.
\end{gather}
\end{lemma}

\begin{proof}
 With the Neumann boundary condition, \eqref{reilly} becomes
\begin{gather}
\int_M\mexhal| L_h u|^2-\exhal \big|D^2u\big|^2 \,\dm\nonumber\\
\qquad{} =\int_M\exhal \big( \Ric+\alpha D^2h\big) (\grad u,\grad u )\,\dm
+\int_{\partial M}\exhal I (\grad_{\partial }u,\grad_{\partial }u )\, \dm.
\label{reilly neum}
\end{gather}
Since $\partial M$ is convex, then $I(\grad_{\partial }u,\grad_{\partial }u) \geqslant 0$ and \eqref{reilly neum} becomes
\begin{gather}
\int_M\mexhal| L_h u|^2-\exhal \big|D^2u\big|^2 \geqslant \int_M\exhal \big( \Ric+\alpha D^2h\big) (\grad u,\grad u )\,\dm \label{eq2} \\
\hphantom{\int_M\mexhal| L_h u|^2-\exhal \big|D^2u\big|^2}{} \geqslant  \alpha^2\int_M\exhal |\grad u|^2\frac{|\grad h|^2}{nz}\,\dm +A\lambda\int_M u^2 \,\dm.\nonumber
\end{gather}
Notice that the same inequality also holds if $\partial M$ is empty.

On the other hand, $\big|D^2u\big|^2\geqslant\frac{|\lap u|^2}{n} $ (see \cite[p.~409]{berger2007panoramic}), and
\begin{gather}
\int_M  \mexhal | L_h u|^2-\exhal |D^2u|^2\,\dm
\leqslant \int_M\mexhal\lambda^2 u^2-\frac{1}{n}\exhal |\lap u|^2 ~\dm\nonumber\\
\qquad{} =\int_M\mexhal\lambda^2 u^2-\frac{1}{n}\exhal (\lambda u-\alpha g(\grad h,\grad u) )^2\,\dm\nonumber\\
\qquad{} \leqslant\int_{M}\mexhal \lambda^2u^2-\frac{1}{n}\exhal \left( \frac{\lambda^2u^2}{z+1}-\alpha^2\frac{|g(\grad h,\grad u)|^2}{z}\right)\,\dm
\nonumber \\
\label{eq2prime}
\qquad{} = \lambda^2\int_{M}{u^2\frac{\mexhal n(z+1) -\exhal }{n(z+1)}~\dm} +\alpha^2\int_M \exhal \frac{|g(\grad h,\grad u)|^2}{nz} \,\dm.
\end{gather}
In the second to last inequality we have used Young's inequality. Indeed, given any $\epsilon>0$,
\begin{equation*}
\lambda u \alpha g(\grad h,\grad u )\leqslant\frac{\lambda^2u^2}{2\epsilon}+\frac{\epsilon}{2}\alpha^2|g(\grad h,\grad u)|^2,
\end{equation*}
 since $\big(\frac{\lambda u}{\sqrt{2\epsilon}}-\sqrt{\frac{\epsilon}{2}}\alpha g(\grad h,\grad u)\big)^2$ is non-negative. Adding the expression
\[ -\frac{1}{2}\big( \lambda^2u^2+\alpha^2|g(\grad h,\grad u)|^2\big)\] to both sides of this inequality, we get
\begin{equation*}
-\big(\lambda u-\alpha g(\grad h,\grad u)\big)^2\leqslant -\left[\lambda^2u^2\left(1-\frac{1}{\epsilon}\right)+\alpha^2|g(\grad h,\grad u)|^2(1-\epsilon)\right].
\end{equation*}
Then choosing $\epsilon:=\frac{z+1}{z}$, one has
\begin{equation*}
- \big(\lambda u-\alpha g(\grad h,\grad u) \big)^2\leqslant -\left[\frac{\lambda^2u^2}{z+1}-\frac{\alpha^2|g(\grad h,\grad u)|^2}{z}\right].
\end{equation*}
Now, combining \eqref{eq2} and \eqref{eq2prime}, we have
\begin{equation*}
A\lambda\int_M u^2 \,\dm \leqslant \lambda^2\int_{M}u^2\frac{\mexhal n(z+1) -\exhal }{n(z+1)} \,\dm.\tag*{\qed}
\end{equation*}\renewcommand{\qed}{}
\end{proof}

Now, we consider $\tilde{\lambda}:= \lambda\big(\frac{|M|}{\int_M {\rm e}^{-h}\dvg} \big)^{\alpha-1}$ which is invariant under rescaling.
Indeed, for any non-zero scalar $a$,
\begin{gather*}
\frac{\int_M|\grad u|^2\big(a{\rm e}^{-h}\big)^\alpha \,\dvg}{\int_M u^2\big(a{\rm e}^{-h}\big)\,\dvg}\cdot\left( \frac{|M|}{\int_M (a{\rm e}^{-h})\,\dvg}\right)^{{\alpha-1}}
=\frac{\int_M|\grad u|^2{\rm e}^{-h\alpha}\, \dvg}{\int_M u^2{\rm e}^{-h}\,\dvg} \left( \frac{|M|}{\int_M {\rm e}^{-h}\,\dvg}\right)^{{\alpha-1}}.
\end{gather*}
Replacing $\lambda$ by $\tilde{\lambda}\big(\frac{\int_M {\rm e}^{-h}\dvg} {|M|}\big)^{\alpha-1} $ in~\eqref{eq3}, under the assumptions of Lemma~\ref{lm07082020}, we get the following inequality:
\begin{equation}\label{eq20062020}
A\tilde{\lambda}\int_M u^2 \,\dm \leqslant\tilde{\lambda}^2\int_M u^2 \left( \frac{\int_M {\rm e}^{-h}\,\dvg} {|M|}\right)^{\alpha-1} \left(\frac{\mexhal n(z+1)-\exhal }{n(z+1)}\right) \dm.
\end{equation}

Let $j\geqslant 2$, $z\in \R_{>0}$, $\alpha>1$ and $|~|\colon M\ni x\longrightarrow d(x_0,x)\in \R_{\geqslant 0}$ where $x_0\in M$ is a fixed point. We define
\begin{gather*} c_0:=\sqrt{n(z+1){\rm e}^{\alpha-1}\left({\rm e}^{\alpha-1}-\frac{1}{j} \right)},\qquad C_j:=-\frac{1}{\alpha-1}\log(c_0), \qquad h_j(x):={\rm e}^{-\frac{|x|^2}{j}}+C_j.
\end{gather*}
 The following properties hold.
\begin{lemma}\label{lem20062020}
The function $h_j$ satisfies:
\begin{enumerate}\itemsep=0pt
\item[$(i)$] $\big( \frac{\int_M {\rm e}^{-h_j}\,\dvg} {|M|}\big)^{\alpha-1}\leqslant c_0$,
\item[$(ii)$] $\frac{{\rm e}^{h_j(\alpha-1)} n(z+1)-{\rm e}^{-h_j(\alpha-1)}}{n(z+1)}\leqslant\frac{1}{jc_0}$,
\item[$(iii)$] $|\grad h_j|^2 -\alpha D^2h_j\leqslant \frac{2\alpha}{j} $.
\end{enumerate}
\end{lemma}
\begin{proof} (i)
$h_j(x)>C_j$ implies that
\[
\int_M {\rm e}^{-h_j(x)}\,\dvg\leqslant\int_Mc_0^{\frac{1}{\alpha-1}}\,\dvg=c_0^{\frac{1}{\alpha-1}}|M|.
\]

(ii) Let us set $b:=n(z+1)$ and $u:={\rm e}^{h_j(\alpha-1)}$.
 We want to prove that
$\frac{(u^2b-1)jc_0-bu}{ubjc_0}\leqslant 0$.
Notice that $u>0$, $bjc_0>0$
 and $\frac{(u^2b-1)jc_0-bu}{u}=\frac{(u-u_1)(u-u_2)}{u}$, where
\[ u_1:=\frac{b-\sqrt{b^2+4bj^2c_0^2}}{2bj c_0}<0 \qquad \text{and} \qquad u_2:=\frac{b+\sqrt{b^2+4bj^2c_0^2}}{2b jc_0}>0.
\]
Moreover,
$0<{\rm e}^{h_j(\alpha-1)} \leqslant u_2 $. Indeed,
${\rm e}^{h_j(\alpha-1)}={\rm e}^{(\alpha-1)\big({\rm e}^{-\frac{|x|^2}{j}}+C_j\big)}$, so the first inequality is immediate. For the second inequality, we have
\[
\sqrt{\frac{4}{b}+\frac{1}{j^2c_0^2}}=
 \frac{1}{c_0} \left( \sqrt{\frac{4c_0^2}{b}+\frac{1}{j^2}}\right)=
 \frac{1}{c_0} \sqrt{\left( 2{\rm e}^{\alpha-1}-\frac{1}{j}\right)^2}= \frac{1}{c_0} \left( 2{\rm e}^{\alpha-1}-\frac{1}{j}\right).
\]
Hence,
\[ \log\left(\frac{1}{c_0}\right)=\log\left[ \frac{1}{2}\left( \sqrt{\frac{4}{b}+\frac{1}{j^2c_0^2}}+\frac{1}{j c_0}\right)\right]-(\alpha-1)\] and
\begin{gather*}
 C_j=\frac{1}{(\alpha-1)} \log\left[ \frac{1}{2}\left( \sqrt{\frac{4}{b}+\frac{1}{j^2c_0^2}}+\frac{1}{j c_0}\right)\right]-1,\\
 h_j(x)\leqslant 1+C_j\leqslant\frac{1}{\alpha-1}\log\left[ \frac{1}{2}\left( \sqrt{\frac{4}{b}+\frac{1}{j^2c_0^2}}+\frac{1}{j c_0}\right)\right].
\end{gather*}
Hence,
\[ {\rm e}^{h_j(\alpha-1)}\leqslant \frac{1}{2}\left( \sqrt{\frac{4}{b}+\frac{1}{j^2c_0^2}}+\frac{1}{j c_0}\right)=u_2.\]

(iii) Setting $r(x)=|x|$, $r$ is radial and we have
\[
h_j(r)={\rm e}^{-\frac{r^2}{j}} , \qquad \grad h_j(r)={\rm e}^{-\frac{r^2}{j}}\left(-\frac{2}{j}\right)r, \qquad D^2h_j(r)=\left(-\frac{2}{j}\right){\rm e}^{-\frac{r^2}{j}}\left(1-\frac{2}{j}\right)r.
\] Hence,
\begin{gather*}
|\grad h_j(r)|^2 -\alpha D_j^2h_j(r) ={\rm e}^{-\frac{r^2}{j}} \left( -\frac{4}{j^2}\alpha r^2 +\frac{4}{j^2}r^2{\rm e}^{-\frac{r^2}{j}}+\frac{2\alpha}{j} \right) \\
\hphantom{|\grad h_j(r)|^2 -\alpha D_j^2h_j(r)}{} ={\rm e}^{-\frac{r^2}{j}} \left( \frac{2\alpha}{j}+\frac{4}{j^2}r^2\big({\rm e}^{-\frac{r^2}{j}} -\alpha \big)\right) \\
\hphantom{|\grad h_j(r)|^2 -\alpha D_j^2h_j(r)}{} \leqslant \frac{2\alpha}{j}, \qquad \text{since} \quad {\rm e}^{-\frac{r^2}{j}}\leqslant 1.\tag*{\qed}
\end{gather*}\renewcommand{\qed}{}
\end{proof}

\begin{proof}[Proof of Theorem \ref{main1}]
We set $z=\frac{\alpha^2}{n}$, $A:=\frac{\kappa}{2}$ and $j_0:=\big\lceil\frac{4\alpha}{\kappa} \big\rceil$. Then from Lemma~\ref{lem20062020}~(iii), we have
\[ \Ric +\alpha D^2h_j\geqslant \kappa+ \alpha^2\frac{|\grad h_j|^2}{nz}-\frac{2\alpha}{j}
\geqslant \alpha^2\frac{|\grad h_j|^2}{nz}+A,\qquad \forall\, j\geqslant j_0.
\]

Combining inequality \eqref{eq20062020}, Lemma \ref{lem20062020} (i) and (ii), we finally get
\[ A\tilde{\lambda}\int_M u^2 \,\dm \leqslant \tilde{\lambda}^2\int_M u^2 c_0 \frac{1}{jc_0}\,\dm.\]
Hence, for every $j\geqslant j_0$, one has
$A j\leqslant\tilde{\lambda} $.
\end{proof}

\subsection*{Acknowledgements}

This paper is part of the author's Ph.D.~Thesis under the direction of Professor Bruno Colbois (Neuch\^atel University). The author wishes to express her thanks to her supervisor for suggesting the problem. She is also grateful to the anonymous referee whose suggestions and remarks greatly improved this paper.

\pdfbookmark[1]{References}{ref}
\LastPageEnding

\end{document}